\documentclass{amsart}
\usepackage{graphicx}
\usepackage{amsfonts}
\usepackage{float}
\newtheorem{tm}{Theorem}

\newtheorem{defi}{Definition}

\newtheorem{rem}{Remark}

\newtheorem{lm}{Lemma}
\newtheorem{ex}{Example}

\newtheorem{prop}{Proposition}
\newtheorem{nota}{Notation}

\newtheorem{prob}{Problem}

\begin{document}
\title{Further than Descartes' rule of signs}
\author{Yousra Gati, Vladimir Petrov Kostov and Mohamed Chaouki Tarchi}

\address{Universit\'e de Carthage, EPT-LIM, Tunisie}
\email{yousra.gati@gmail.com}
\address{Universit\'e C\^ote d’Azur, CNRS, LJAD, France}
\email{vladimir.kostov@unice.fr}
\address{Universit\'e de Carthage, EPT-LIM, Tunisie}
\email{mohamedchaouki.tarchi@gmail.com}
\begin{abstract}
  The {\em sign pattern} defined by the real polynomial
  $Q:=\Sigma _{j=0}^da_jx^j$, $a_j\neq 0$, is the
  string $\sigma (Q):=({\rm sgn(}a_d{\rm )},\ldots ,{\rm sgn(}a_0{\rm )})$.
  The quantities
  $pos$ and $neg$ of positive and negative roots of $Q$ satisfy Descartes'
  rule of signs. A couple $(\sigma _0,(pos,neg))$, where $\sigma _0$ is a
  sign pattern of length $d+1$, is {\em realizable} if
  there exists a polynomial $Q$ with $pos$ positive and $neg$ negative
  simple roots, with $(d-pos-neg)/2$ complex conjugate pairs and with
  $\sigma (Q)=\sigma_0$. We present a series of couples (sign pattern,
  pair $(pos,neg)$) depending on two integer parameters and with
  $pos\geq 1$, $neg\geq 1$, which is not realizable.
  For $d=9$, we give the exhaustive list of realizable couples with
  two sign changes in the sign pattern.

  {\bf Key words:} real polynomial in one variable; hyperbolic polynomial;
  sign pattern; Descartes' rule of signs\\

{\bf AMS classification:} 26C10; 30C15
\end{abstract}
\maketitle

\section{Introduction}

This paper deals with a problem which is a natural extension of
Descartes' rule of signs. We consider univariate real polynomials
without vanishing
coefficients. About such a degree $d$ polynomial $Q$, Descartes' rule of
signs states that the number $pos$ of its positive roots is bounded by the
number $c$ of sign changes in the sequence of its coefficients, the
difference $c-pos$ being even, see \cite{Ca}, \cite{Cu}, \cite{DG},
\cite{Des}, \cite{Fo}, \cite{Ga}, \cite{J}, \cite{La} or \cite{Mes}.
When applied to the polynomial $Q(-x)$, this
rule yields the result that the number $neg$ of its negative roots is bounded
by the number $p$ of sign preservations and the difference $p-neg$
is also even. In the case of {\em hyperbolic} polynomials, i.e. real
polynomials with all roots real, one has $pos=c$ and $neg=p$. Descartes' rule
of signs proposes only necessary conditions. Our aim is to understand how far
from sufficient they can be.

In what follows we use the quantity
$\tilde{\mu}:=\min (pos, neg)$ and we denote by $\tilde{\lambda}$ the number of
complex conjugate pairs of roots of $Q$. Hence the quantities which we 
introduce 
satisfy the following relations in which $[.]$ denotes the integer part:

\begin{equation}\label{eqdefi}
  \begin{array}{ccc}pos\leq c~,&neg\leq p~,&\tilde{\lambda}=(d-pos-neg)/2~,\\ \\
    c-pos\in 2\mathbb{Z}~,&p-neg\in 2\mathbb{Z}~,&
    \tilde{\lambda}\leq [d/2]~,\\ \\ {\rm sgn}(Q(0))=(-1)^{pos}&{\rm and}&
    c+p=d~. 
    \end{array}
  \end{equation}

\begin{defi}
  {\rm (1) A {\em sign pattern} of length $d+1$ is a string of $d+1$ signs $+$
    and/or $-$. We say that the polynomial $Q:=\sum _{j=0}^{d}a_jx^j$
    defines the sign pattern $\sigma (Q):=({\rm sgn}(a_d)$, $\ldots$,
    ${\rm sgn}(a_0))$. Most often we deal with monic polynomials in which
    case the first sign of the sign pattern is a~$+$.
    \vspace{1mm}
    
    (2) Suppose that a sign pattern $\sigma _0$ is given, with $c$
    sign changes
    and $p$ sign preservations. An {\em admissible pair}
    (i.~e. a pair admissible for $\sigma _0$)
    is a pair $(pos, neg)$ satisfying conditions (\ref{eqdefi}). In this case
    we say that $(\sigma _0,(pos, neg))$ is a {\em compatible couple} (or just
    {\em couple} for short).
    \vspace{1mm}
    
    (3) We say that this compatible couple is {\em realizable}
    if there exists a
    degree $d$ real polynomial $Q$ with $\sigma (Q)=\sigma _0$, with exactly
    $pos$ positive and $neg$ negative roots, all of them distinct.
    \vspace{1mm}
    
    (4) The pair $(pos, neg)=(c,p)$, which is admissible for the sign pattern
    $\sigma _0$, is its
    {\em Descartes' pair}.}
  \end{defi}

We study the following problem:

\begin{prob}\label{prob1}
  For a given degree $d$, which compatible couples are realizable?
\end{prob}

In spite of its simple formulation the problem is not trivial at all.
Descartes' rule of signs gives only necessary conditions,
and Problem~\ref{prob1} is
a realization problem. For $d=1$, $2$ and $3$, all compatible couples
are realizable, but for $d=4$, the couple $((+,-,-,-,+),(0,2))$ is not
(see~\cite{Gr}). In fact, for $d=4$, this is the only non-realizable couple
up to the following equivalence.

\begin{defi}\label{defiZ2Z2}
  {\rm For a given degree $d$, the
    $\mathbb{Z}_2\times \mathbb{Z}_2$-{\em action} on the set
    of compatible couples is defined by means of two commuting involutions.
    The involution}

  \begin{equation}\label{eqim}
    i_m~:~(\sigma (Q), (pos, neg))\mapsto
    ((-1)^d\sigma (Q(-x)), (neg, pos))
    \end{equation}
  {\rm expresses the fact that the change of variable $x\mapsto -x$ changes
    the sign of every second coefficient and the signs of the real roots.
    The involution}

  \begin{equation}\label{eqir}
    i_r~:~(\sigma (Q), (pos, neg))\mapsto
    (\sigma (Q^R(x)/Q(0))(pos, neg))~,~~~\, Q^R(x):=x^dQ(1/x)
    \end{equation}
      {\rm is connected with the property the {\em reverted} polynomial $Q^R$
        (i.~e. $Q$ read from right to left) to have the same numbers
    of positive and negative roots as $Q$ (reversion changes the roots to
    their reciprocals). The normalizing factors $(-1)^d$ and
    $1/Q(0)$ are introduced to preserve the set of monic polynomials.
    The {\em orbit} of each compatible couple under this action
    consists either of $2$ or of $4$ couples which share the same values of
    $\tilde{\lambda}$ and $\tilde{\mu}$ and which are simultaneously
    realizable or not. This is why we use the same notation for couples and
    for their orbits. Sometimes we consider the orbits under the
    $\mathbb{Z}_2\times \mathbb{Z}_2$-action only of sign patterns,
    not of couples.}
  \end{defi}

\begin{nota}
  {\rm We denote by $\Sigma _{m_1,m_2,\ldots ,m_s}$, $m_k\in \mathbb{N}$,
    $m_1+\cdots +m_s=d+1$, the sign pattern beginning with a sequence of $m_1$
    signs $+$ followed by a sequence of $m_2$ signs $-$
    followed by a sequence
    of $m_3$ signs $+$ etc. Example:}
    $$(+,+,-,-,-,+,-,+,+,+)=\Sigma _{2,3,1,1,3}~.$$
  \end{nota}

\begin{ex}
  {\rm For $d=3$, the orbit of the sign pattern $\sigma ^{\diamond}:=(+,+,-,+)$
    consists of four sign patterns: $\sigma ^{\diamond}$,
    $i_r(\sigma ^{\diamond})=(+,-,+,+)$, $i_m(\sigma ^{\diamond})=(+,-,-,-)$
    and $i_r(i_m(\sigma ^{\diamond}))=(+,+,+,-)$. The orbit of the sign pattern
    $\Sigma _{1,d-1,1}$ ($d\geq 3$) consists of two sign patterns:}

    $$\Sigma _{1,d-1,1}=i_r(\Sigma _{1,d-1,1})~~~\, {\rm and}~~~\, 
    i_m(\Sigma _{1,d-1,1})=i_m(i_r(\Sigma _{1,d-1,1}))=\Sigma _{2,1,\ldots ,1,2}$$ 
    {\rm ($d-3$ units).
    For $d$ even, both
    $\Sigma _{1,d-1,1}$ and  $i_m(\Sigma _{1,d-1,1})$ are center-symmetric.
    For $d$ odd, $\Sigma _{1,d-1,1}$ is center-symmetric and
    $i_m(\Sigma _{1,d-1,1})$ is center-antisymmetric.}
\end{ex}

\begin{rem}{\rm For an orbit of length $2$, one has either
    $i_r(\sigma (Q))=\sigma (Q)$ or $i_r(i_m(\sigma (Q)))=\sigma (Q)$. One has
    always $i_m(\sigma (Q))\neq \sigma (Q)$, because the second components of
    these two sign patterns are different.}\end{rem}

Problem~\ref{prob1} is formulated for the first time in \cite{AJS}.
Up to the $\mathbb{Z}_2\times \mathbb{Z}_2$-action, for $d=5$ and $d=6$, 
there exist exactly $1$ and exactly $4$ non-realizable
orbits respectively,
see~\cite{AlFu}. For $d=7$ (see \cite{FoKoSh}) and $d=8$ (see \cite{FoKoSh}
and \cite{KoCzMJ}) these numbers equal $6$ and $19$. In all these cases
one has $\tilde{\mu}=0$ and the exhaustive answer to Problem~\ref{prob1}
is known.
Details are given in Section~\ref{secprelim}.

For hyperbolic polynomials, the answer to Problem~\ref{prob1}
is always positive, see \cite[Proposition~1]{KoSe}. In other words,
the orbit of any compatible couple in which the admissible pair is the
Descartes' pair
is realizable. A tropical analog of
Descartes' rule of signs is studied in~\cite{FoNoSh}.

The first examples of non-realizable orbits with $\tilde{\mu}=1$ are found
for $d=11$ in \cite{KoMB} and for $d=9$ in \cite{CGK}. For $d=9$, the orbit
studied in
\cite{CGK} is the only non-realizable one with $\tilde{\mu}=1$.
In this paper we present a series
of non-realizable orbits (depending on two integer
parameters) with $\tilde{\mu}=1$.
The series includes the examples for $d=11$ and
$d=9$ already found. Besides, for $d=10$ and $d=11$,
we give the list of the orbits
with $\tilde{\mu}=1$ which are either non-realizable or for which the answer to
Problem~\ref{prob1} is not known. We also show that
for $d\leq 14$, with the possible exception of one orbit with $d=14$,
there are no non-realizable orbits with $\tilde{\mu}\geq 2$. Finally,
we give the exhaustive answer to Problem~\ref{prob1} for $d=9$, $c=2$. 

\begin{defi}
  {\rm For $d\geq 9$, we denote by $\mathcal{K}_{n,q}$ the
    orbit $(\Sigma _{1,n,q,1},(1,d-3))$, $n\geq 4$, $q\geq 4$, 
    $n+q=d-1$, hence with Descartes' pair $(3,d-3)$. The orbits
    $\mathcal{K}_{n,q}$ and $\mathcal{K}_{q,n}$
    being the same one can assume that $n\leq q$.}
\end{defi}
  
\begin{tm}\label{tmmain}
  For $d\geq 9$, the orbit $\mathcal{K}_{n,q}$ is not realizable. 
\end{tm}

The theorem is proved in Section~\ref{secprtmmain}.

\begin{tm}\label{tmmainbis}
  (1) For $d\leq 14$, with the possible exception of the orbit 
  $(\Sigma _{1,4,5,4,1},(2,10))$
  about which the answer to Problem~\ref{prob1} is not known, there are
  no non-realizable orbits with $\tilde{\mu}\geq 2$.
  \vspace{1mm}

  (2) For $d=10$, for the orbits $\mathcal{K}_{4,5}=(\Sigma _{1,4,5,1},(1,7))$
  and $(\Sigma _{1,4,4,2},(1,7))$, one has
  $(\tilde{\mu},\tilde{\lambda})=(1,1)$. The first
  of them is not realizable (see Theorem~\ref{tmmain}), for the second one
  the answer to Problem~\ref{prob1} is not known. All other orbits with
  $d=10$ and $\tilde{\mu}=1$ are realizable.
  \vspace{1mm}
%
    
    (3) For $d=11$, the following two orbits (both with
    $\tilde{\mu}=\tilde{\lambda}=1$)
    are not realizable:

    $$\mathcal{K}_{5,5}=(\Sigma _{1,5,5,1},(1,8))~~~\, {\rm and}~~~\, 
      \mathcal{K}_{4,6}=(\Sigma _{1,4,6,1},(1,8))~.$$
    For $d=11$, with the exception of these two and of the following six
    orbits
    $$\begin{array}{cl}
      (\Sigma _{1,4,5,2},(1,8))~,~~~(\tilde{\mu},\tilde{\lambda})=(1,1)~,&
      (\Sigma _{1,5,4,2},(1,8))~,~~~(\tilde{\mu},\tilde{\lambda})=(1,1)~,\\ \\ 
      (\Sigma _{2,4,4,2},(1,8))~,~~~(\tilde{\mu},\tilde{\lambda})=(1,1)~,&
      (\Sigma _{1,4,4,1,1,1},(1,6))~,(\tilde{\mu},\tilde{\lambda})=(1,2)~,\\ \\ 
      (\Sigma _{1,4,1,1,4,1},(1,6))~,~~~(\tilde{\mu},\tilde{\lambda})=(1,2)~,&
      (\Sigma _{1,3,1,1,1,1,3,1},(1,4))~,~~~(\tilde{\mu},\tilde{\lambda})=(1,3)~,
    \end{array}$$
    (about which the answer to
    Problem~\ref{prob1} is not known), all other non-realizable orbits
    are with $\tilde{\mu}=0$.
    \end{tm}

The theorem is proved in Section~\ref{secprtmmainbis}. Our next result
concerns orbits of sign patterns with two sign changes.
Sign patterns with none or one sign
change are realizable, see~\cite[Example~1 and Theorem~1]{CGK1},
and the first examples of
non-realizability are obtained for $c=2$, see the beginning of
Section~\ref{secprelim}. 

\begin{tm}\label{tm2SC}
  For $d=9$, the following orbits are not realizable:

  $$\begin{array}{l}
    (\Sigma _{1,8,1},(0,k))~,~~~\, k=3~,~5~~~\, {\rm and}~~~\, 7~;\\ \\
    (\Sigma _{1,7,2},(0,\ell ))~,~~~\, (\Sigma _{1,6,3},(0,\ell ))~~~\,
    {\rm and}~~~\, (\Sigma _{2,6,2},(0,\ell ))~,~~~\,
    \ell =5~~~\, {\rm and}~~~\, 7~;\\ \\ 
    (\Sigma _{1,5,4},(0,7))~,~(\Sigma _{1,4,5},(0,7))~,~
    (\Sigma _{2,4,4},(0,7))~,~(\Sigma _{2,5,3},(0,7))~~~\, {\rm and}~~~\,
    (\Sigma _{3,4,3},(0,7))~.
  \end{array}$$
  All other orbits $(\Sigma _{m,n,q},(pos,neg))$ with $m>0$, $n>0$, $q>0$,
  $m+n+q=10$, are realizable.
  \end{tm}

The theorem is proved in Section~\ref{secprtm2SC}.

\section{Preliminaries\protect\label{secprelim}}

We list here the non-realizable orbits for $4\leq d\leq 8$. In all
cases one has $\tilde{\mu}=0$. For $d=4$ and $d=5$, 
the non-realizable orbits are:

\begin{equation}\label{eqd4d5}
 (\Sigma _{1,3,1},(0,2))~~~~\, {\rm and}~~~\, 
 (\Sigma _{1,4,1},(0,3))~.
  \end{equation}
For $d=6$, there are four such orbits:

\begin{equation}\label{eqd=6}
    (\Sigma _{1,5,1},(0,2))~,~~~\, (\Sigma _{1,5,1},(0,4))~,~~~\, 
    (\Sigma _{1,1,1,3,1},(0,2))~~~\, {\rm and}~~~\, (\Sigma _{2,4,1},(0,4))~.
  \end{equation}
For $d=7$, there are six non-realizable orbits:

\begin{equation}\label{eqd=7}
  \begin{array}{lll}
    (\Sigma _{2,5,1},(0,5))~,&(\Sigma _{2,4,2},(0,5))~,&
    (\Sigma _{3,4,1},(0,5))~,\\ \\ 
    (\Sigma _{1,4,1,1,1},(0,3))~,&
      (\Sigma _{1,6,1},(0,3))~,&(\Sigma _{1,6,1},(0,5))~. 
  \end{array}
  \end{equation}
For $d=8$, the non-realizable orbits are the following ones:

\begin{equation}\label{eqd=8}
  \begin{array}{l}
    \Sigma _{2,5,2}~,~~~\, \Sigma _{1,6,2}~,~~~\, \Sigma _{1,4,4}~,~~~\,
    \Sigma _{1,5,3}~~~\, {\rm and}~~~\, \Sigma _{2,4,3}~~~\,
           {\rm with}~~~\, (0,6)~,\\ \\ 
           \Sigma _{1,1,1,3,1,1,1}~~~\, {\rm and}~~~\, \Sigma _{1,3,1,1,1,1,1}~~~\,
                  {\rm with}~~~\, (0,2)~,\\ \\
                  \Sigma _{1,5,1,1,1}~~~\, {\rm and}~~~\, \Sigma _{1,3,1,3,1}~~~\,
                         {\rm with}~~~\, (0,2)~~~\, {\rm and}~~~\, (0,4)~,\\ \\
                         \Sigma _{1,7,1}~~~\, {\rm with}~~~\, (0,2)~,~~~\,
                         (0,4)~~~\, {\rm and}~~~\, (0,6)~,\\ \\ 
                         \Sigma _{1,4,1,2,1}~,~~~\,  \Sigma _{1,6,2}~,~~~\,
                         \Sigma _{1,4,2,1,1}~,~~~\, \Sigma _{1,1,1,4,2}
                         ~~~\, {\rm and}~~~\, \Sigma _{1,4,1,1,2}~~~\,
                  {\rm with}~~~\, (0,4)~.
    \end{array}
  \end{equation}

The following {\em concatenation lemma} (proved in \cite{FoKoSh})
is a basic tool for proving the relizability of certain orbits.

\begin{lm}\label{lmconcat}
Suppose that the
monic polynomials $P_1$ and $P_2$ of degrees $d_1$ and $d_2$, with
sign patterns represented in the form 
$(+,\sigma _1)$ and $(+,\sigma _2)$ respectively, realize
the pairs $(pos_1, neg_1)$ and $(pos_2, neg_2)$. Here $\sigma _j$
denotes what remains of the sign patterns when the initial sign $+$ is deleted.
Then

(1) if the last position of $\sigma _1$ is $+$, then for any $\varepsilon >0$
small enough, the polynomial $\varepsilon ^{d_2}P_1(x)P_2(x/\varepsilon )$
realizes the sign pattern $(+,\sigma _1,\sigma _2)$ and the admissible pair
$(pos_1+pos_2, neg_1+neg_2)$;

(2) if the last position of $\sigma _1$ is $-$, then for any $\varepsilon >0$
small enough, the polynomial $\varepsilon ^{d_2}P_1(x)P_2(x/\varepsilon )$
realizes the sign pattern $(+,\sigma _1,-\sigma _2)$ and the pair
$(pos_1+pos_2, neg_1+neg_2)$. Here $-\sigma _2$ is obtained from $\sigma _2$
by changing each $+$ by $-$ and vice versa.
\end{lm}

\begin{rem}
  {\rm We use the symbol $\ast$ to denote concatenation of couples or of
    sign patterns.
    Lemma~\ref{lmconcat} implies that
    when one concatenates the compatible realizable couples
    $(\Sigma _{m_1,\ldots ,m_s},(a,b))$ and $(\Sigma _{n_1,\ldots ,n_{\ell}},(c,d))$
    one obtains the realizable couple}
  $$(\Sigma _{m_1,\ldots ,m_{s-1},m_s+n_1-1,n_2,\ldots ,n_{\ell}},(a+c,b+d))=
  ((\Sigma _{m_1,\ldots ,m_s},(a,b))\ast (\Sigma _{n_1,\ldots ,n_{\ell}},(c,d))~.$$
  {\rm If one considers only sign patterns instead of couples, then
    one can write}
  $$\Sigma _{m_1,\ldots ,m_{s-1},m_s+n_1-1,n_2,\ldots ,n_{\ell}}=
  \Sigma _{m_1,\ldots ,m_s}\ast \Sigma _{n_1,\ldots ,n_{\ell}}~.$$
         {\rm When necessary we use more than two consecutive concatenations.}
\end{rem}

Consider a hyperbolic monic degree $d$ polynomial without vanishing
coefficients. Suppose that the moduli of
its roots are all distinct. Consider the order of these moduli on the real
positive half-axis. We note only at which positions the moduli
of its negative roots are; this should be clear from the following example.

\begin{ex}
  {\rm The polynomial}

  $$(x-1)(x+2)(x-3)(x-4)(x+5)(x-6)(x+7)(x+8)(x-9)$$
  {\rm has five
    positive and four negative roots. We note the relative
    positions of the moduli of its positive
    and negative roots by the letters $P$ and $N$. The order of the
    moduli of the roots of the polynomial is}
  $$|1|<|-2|<|3|<|4|<|-5|<|6|<|-7|<|-8|<|9|$$
  {\rm which we note as}
  $$P<N<P<P<N<P<N<N<P~.$$
\end{ex}

Given a sign pattern $(\alpha _d,\alpha _{d-1},\ldots ,\alpha _0)$,
$\alpha _j=\pm 1$,
one can construct a hyperbolic
degree $d$ polynomial defining
this sign pattern using Lemma~\ref{lmconcat}. At the first step one constructs
the linear polynomial $P_+:=x+1$ if $\alpha _d=\alpha _{d-1}$ or $P_-:=x-1$ if
$\alpha _d=-\alpha _{d-1}$. At each next step one concatenates the previously
obtained polynomial (which plays the role of $P_1$ and which defines the sign
pattern $(\alpha _d,\alpha _{d-1},\ldots ,\alpha _j)$) and the polynomial $P_+$
or $P_-$ as $P_2$ depending on whether $\alpha _j=\alpha _{j-1}$ or
$\alpha _j=-\alpha _{j-1}$ respectively.

Hence the modulus of each next root
which is added during this construction is much smaller than the moduli of the
roots previously obtained. In the end the roots of the constructed
polynomial define the {\em canonical order} corresponding to the given sign
pattern: one reads the sign pattern
from the right, to each consecutive equal (resp. different)
signs of coefficients
puts in correspondence the letter $N$ (resp. $P$) and then inserts between
any two consecutive letters the sign $<$. E.~g. the sign pattern
$(+,-,-,+,-,+,+,+,-)$ defines the canonical order

$$P<N<N<P<P<P<N<P~.$$
Thus for every sign pattern, there exists a hyperbolic polynomial the moduli
of whose roots define the corresponding canonical order,
see~\cite[Proposition~1]{KoSe}.

\section{Proof of Theorem~\protect\ref{tmmain}\protect\label{secprtmmain}}

\begin{proof}
  A) Suppose that a couple $\mathcal{K}_{n,q}$ is realizable by some polynomial
  $Q$.
  Using if necessary a linear transformation $x\mapsto hx$, $h>0$, one can
  assume that one of the roots of $Q$ is at~$1$. (We remind that $Q$ has three
  sign changes in the sequence of its coefficients, so by Descartes' rule
  of signs $Q$ has either
  $3$ or $1$ positive roots.) We denote by $-x_i$ the
  negative roots of $Q$, $0<x_1\leq x_2\leq \cdots \leq x_{d-3}$, and we set

  $$\sum _{j=1}^da_jx^j=:Q:=
  (x^{d-3}+e_1x^{d-4}+\cdots +e_{d-4}x+e_{d-3})(x^2-ux+v)(x-1)~,$$
  where $e_k$ are the elementary symmetric polynomials of the
  quantities~$x_i$. Hence $e_k>0$ and $a_d=1$. We explicit
  some of the first and some of the last coefficients~$a_j$:

  \begin{equation}\label{eqlist}
    \begin{array}{lrrrrrr}
      a_{d-1}&=&(e_1-1)&-&u~,&&\\ \\
      a_{d-2}&=&(e_2-e_1)&-&(e_1-1)u&+&v~,\\ \\
      a_{d-3}&=&(e_3-e_2)&-&(e_2-e_1)u&+&(e_1-1)v~,\\ \\
      a_{d-4}&=&(e_4-e_3)&-&(e_3-e_2)u&+&(e_2-e_1)v~,\\ \\ 
      a_3&=&(e_{d-3}-e_{d-4})&-&(e_{d-4}-e_{d-5})u&+&(e_{d-5}-e_{d-6})v~,\\ \\ 
      a_2&=&-e_{d-3}&-&(e_{d-3}-e_{d-4})u&+&(e_{d-4}-e_{d-5})v~,\\ \\ 
      a_1&=&&&e_{d-3}u&+&(e_{d-3}-e_{d-4})v~,\\ \\ 
      a_0&=&&&&&-e_{d-3}v~.
    \end{array}
  \end{equation}

  B) we introduce some notation:
  
  \begin{nota}
  {\rm In the plane of the variables $(u,v)$ we consider the parabola
  $\mathcal{P}:v=u^2/4$ and the straight lines $(a_j)$ defined by the
  respective equations of the form $a_j=\cdots$ from the list (\ref{eqlist}).
  We set}
  $$\begin{array}{lll}
    c_-:=23-4\sqrt{30}=1.09\ldots~,&c_+:=23+4\sqrt{30}=44.90\ldots ~,&
  I:=(c_-,c_+)~,\\ &&{\rm (see~(\ref{eq49})}\\ E_1:=\sum _{\nu =1}^{d-3}1/x_{\nu}~,&
  E_2:=\sum _{1\leq i<j\leq d-3}1/(x_ix_j)~.&\end{array}$$
  {\rm Hence $e_{d-4}=e_{d-3}\cdot E_1$ and
 $e_{d-5}=e_{d-3}\cdot E_2$.}
  \end{nota}

  \begin{rem}
    {\rm Points above (resp. below) the parabola $\mathcal{P}$
      correspond to polynomials $x^2-ux+v$ having two complex conjugate (resp.
      two real distinct) roots. Any polynomial from the parabola has
      a double real root. Polynomials between the parabola and the $u$-axis
      have two positive roots if $u>0$ and two negative roots if $u<0$.
      Polynomials below the $u$-axis have two
      roots of opposite signs. Our aim is to show that the domain defined by
      the inequalities of the form $a_j>0$ or $a_j<0$ resulting from the sign
      pattern $\Sigma _{1,n,q,1}$ does not intersect the domain $\{ v>u^2/4\}$.
    This is why in what follows we assume that $v>0$.}
  \end{rem}
  
  We consider first the case $e_1>c_-$. As $e_1>1$ and $a_{d-1}<0$, one has
  $u>0$ (see (\ref{eqlist})).

  For $e_1\in I$, from the definition of $e_1=x_1+\cdots +x_{d-3}$
  and $E_1$ follows that

  \begin{equation}\label{eqeE}
    \left\{ \begin{array}{l}
      e_1\cdot E_1\geq (d-3)^2\geq 49~~~\, {\rm hence}~~~\,
      E_1\geq 49/c_+>1~~~\, {\rm and}\\ \\ e_{d-4}=e_{d-3}\cdot E_1>e_{d-3}~.
      \end{array} \right. 
    \end{equation}

  \begin{lm}\label{lmI}
    (1) For $e_1\in I$, the intersection point $S$ of the straight lines
    $(a_1)$ and $(a_{d-1})$
    is below the parabola $\mathcal{P}$. 

    (2) Suppose that $e_1\geq c_+$ and that $x_1\leq 1$. Then the 
    point $S$ 
    is below the parabola~$\mathcal{P}$.
    \end{lm}
  
  \begin{proof}[Proof of Lemma~\ref{lmI}]
      Part (1). One has (see (\ref{eqlist}))

      $$(a_1)\cap (a_{d-1})~=~S~:=~(~e_1-1~,~e_{d-3}(e_1-1)/(e_{d-4}-e_{d-3})~)~$$
      and $e_{d-4}>e_{d-3}$, see (\ref{eqeE}).
      The point $S$ is below the parabola $\mathcal{P}$. This follows from

      $$e_{d-3}(e_1-1)/(e_{d-4}-e_{d-3})<(e_1-1)^2/4$$
      which is equivalent to

      \begin{equation}\label{eqed34}
        e_{d-3}(e_1+3)<e_{d-4}(e_1-1)
      \end{equation}
      or to $e_1+3<(e_1-1)E_1$. However (see (\ref{eqeE}))

      $$(e_1-1)E_1\geq (e_1-1)(49/e_1)=49-49/e_1>e_1+3~;$$
      the second of these inequalities is equivalent to

      \begin{equation}\label{eq49}
        e_1^2-46e_1+49=(e_1-c_-)(e_1-c_+)<0~,~~~\, 
        {\rm i.~e.~to}~~~\, e_1\in I~.
        \end{equation}

      Part (2). For fixed sum $x_2+\cdots +x_{d-3}$, the sum
      $\sum _{\nu =2}^{d-3}1/x_{\nu}$ is minimal if $x_2=\cdots =x_{d-3}$. Hence for
      $x_1\leq 1$, one has

      $$E_1\geq 1+(d-4)/(e_1/(d-4))=1+(d-4)^2/e_1\geq 1+36/e_1~,$$
      so again $e_{d-4}>e_{d-3}$. 
      Inequality (\ref{eqed34}) can be given the equivalent form

      $$1+4/(e_1-1)<e_{d-4}/e_{d-3}=E_1~.$$
      However the inequalities $E_1\geq 1+36/e_1>1+4/(e_1-1)$ hold
      true for $e_1\geq c_+$
      from which part (2) of the lemma follows.
    \end{proof}

  There exist no couples $\mathcal{K}_{n,q}$ satisfying the assumptions of
  Lemma~\ref{lmI} since the lemma implies that 
  the domain defined by the
  inequalities

  $$\begin{array}{lll}
    v>u^2/4~,&&a_1=e_{d-3}u+(e_{d-3}-e_{d-4})v>0~~~\, {\rm and}\\ \\ 
  a_{d-1}=e_1-1-u<0~,&{\rm i.~e.}&u>e_1-1~,\end{array}$$
  is void. Indeed, the straight line $(a_1)$ has a positive slope, see
  (\ref{eqeE}); it intersects the parabola $\mathcal{P}$ at the origin and
  at a point with $u>0$. Hence the whole sector $\{ a_1>0,~u>e_1-1\}$ is below
  the parabola~$\mathcal{P}$. 
\vspace{1mm}

C) Suppose that $x_{\nu}\geq 1$, $1\leq \nu \leq d-3$, so $e_1>c_-$.
We consider the intersection point 

$$T~:=~(a_{d-4})\cap (a_{d-1})~=(~e_1-1~,~
((e_3-e_2)(e_1-1)-(e_4-e_3))/(e_2-e_1)~)~.$$
Observe first that for $d\geq 10$ and $x_{\nu}\geq 1$, one has $e_2-e_1>0$
and $e_3-e_2>0$.
We show that the point $T$ is below the parabola $\mathcal{P}$. The straight
line $(a_{d-4})$ has a positive slope. Hence the domain defined by the three
inequalities 

$$\begin{array}{lll}
  v>u^2/4~,&&a_{d-1}=e_1-1-u<0~~~\, {\rm and}\\ \\
  a_{d-4}<0~,&{\rm i.~e.}&v<((e_3-e_2)u-(e_4-e_3))/(e_2-e_1)~,\end{array}$$
is void, so there exist no couples $\mathcal{K}_{n,q}$ with $e_1\geq c_+$ and
$x_{\nu}\geq 1$.

The point $T$ is under the parabola $\mathcal{P}$ exactly when

$$\begin{array}{l}
  ((e_3-e_2)(e_1-1)-(e_4-e_3))/(e_2-e_1)<(e_1-1)^2/4~,~~~\, {\rm i.~e.}\\ \\
\Psi:=4(e_3-e_2)(e_1-1)-4(e_4-e_3)-(e_1-1)^2(e_2-e_1)<0~.
\end{array}$$
We consider the quantity $\Psi$ as a function of one of the variables $x_{\nu}$
(say, $x_1$) when the other variables $x_{\nu}$ are fixed. We denote here
$de_j/dx_1$ by $e_j'$. Clearly $e_j'=f_{j-1}$, where $f_{j-1}$ is the $(j-1)$st
elementary symmetric polynomial of the quantities $x_2$, $\ldots$, $x_{d-3}$;
we set $e_0=f_0=1$.
Thus $e_j=f_j+x_1f_{j-1}$, 

$$\begin{array}{l}
  \Psi =e_1^3-e_1^2e_2-2e_1^2-2e_1e_2+4e_1e_3+e_1+3e_2-4e_4~~~\, {\rm and}\\ \\
\Psi '=3e_1^2-2e_1e_2-e_1^2f_1-4e_1-2f_1e_1-2e_2+4f_2e_1+4e_3+1+3f_1-4f_3~.
  \end{array}$$
Substituting $f_j+x_1f_{j-1}$ for $e_j$ in $\Psi '$ gives

$$\begin{array}{cclcl}\Psi '&=&-(f_1+3x_1-1)(f_1^2+x_1f_1-2f_2-x_1+1)&&\\ \\ &=&
-(f_1+3x_1-1)(x_1(f_1-1)+f_1^2-2f_2+1)&<&0~,\end{array}$$
because $f_1>1$ and $f_2<f_1^2/2$. Thus $\Psi '<0$. Hence if one considers
$\Psi$ as a function in all the variables $x_j$, one finds that
$\partial \Psi /\partial x_j<0$. Hence $\Psi$ takes its maximal value for
$x_1=\cdots =x_{d-3}=1$. In this case $e_j={d-3 \choose j}$ and

$$\Psi =-(d-2)(d-3)(d-4)/2<0~.$$
Thus for $x_{\nu}\geq 1$, one has $\Psi <0$, $e_1>c_-$ and
the point $T$ is below the parabola $\mathcal{P}$.
\vspace{1mm}

D) Up to now we showed that there are no couples $\mathcal{K}_{n,q}$
for $e_1>c_-$.
For $e_1\in I$, this was deduced from part (1) of Lemma~\ref{lmI}; for
$e_1\geq c_+$, this follows from part (2) of Lemma~\ref{lmI} (if $x_1\leq 1$)
and from C) (if $x_1>1$).
Suppose now that $0<e_1\leq c_-$. The involution $i_r$ (see (\ref{eqir}))
transforms the
polynomial $Q$ into a polynomial with sign pattern $\Sigma _{1,q,n,1}$ and with
$e_1>c_-$; the factor $x-1$ is preserved and each of the other two factors
of $Q$ is replaced by a factor of the same form.
For $e_1\leq c_-$, at least $d-4$ of the quantities $x_{\nu}$
are $<1$, so after applying the involution $i_r$ they become $1/x_{\nu}>1$
and $e_1$
becomes larger than $d-4\geq 6$. Thus the proof of
Theorem~\ref{tmmain} for $0<e_1\leq c_-$ results directly from its proof for
$e_1>c_-$.
    
  \end{proof}

\section{Proof of Theorem~\protect\ref{tmmainbis}
  \protect\label{secprtmmainbis}}



A) The proof of Theorem~\ref{tmmainbis} is organised as follows. 
The proof of part (1) is given 
for each degree from $10$ to $14$ in B) -- F)
respectively. Parts B) and C) of the proof are subdivided into B1), C1)
(in which we prove part (1) of the theorem for $d=10$ and $d=11$) and B2), C2)
containing the proofs of parts (2) and (3) of the theorem. 
Using the $\mathbb{Z}_2\times \mathbb{Z}_2$-action
(see Definition~\ref{defiZ2Z2})
one can assume that $pos\leq neg$. 
We remind that:
\vspace{1mm}

(i) For $d\leq 8$,
there are no non-realizable orbits with $\tilde{\mu}\geq 1$
(see \cite{AlFu}, \cite{FoKoSh} and~\cite{KoCzMJ}).
\vspace{1mm}

(ii) For $d=9$, the only
non-realizable orbit with $\tilde{\mu}\geq 1$ is (see~\cite{CGK})

\begin{equation}\label{eqcased9}
  (\Sigma _{1,4,4,1},(1,6))~,~~~{\rm with}~~~\, \tilde{\mu}=1~~~\, {\rm and}~~~\,
  \tilde{\lambda}=1~.
\end{equation}

(iii) If the admissible pair of a given orbit is
$(0,0)$, $(0,1)$, $(1,0)$ or $(1,1)$, then the orbit is realizable. Indeed, if
the admissible pair equals $(0,0)$ or $(0,1)$ (resp. $(1,0)$), then one chooses
a polynomial with the given sign pattern $\sigma$
and adds to it a sufficiently large
positive (resp. negative) constant. If the admissible pair is $(1,1)$, then
one represents the sign pattern in the form
$\sigma =(\sigma _{\dagger},\alpha ,\beta )$, where
$\alpha$ and $\beta$ are its last two signs of $\sigma$.
If $\alpha =\beta$ (resp.
$\alpha =-\beta$), then one
uses the concatenation
$(\sigma _{\dagger},\alpha ),(1,0))\ast (\Sigma _{2},(0,1))$
(resp. $(\sigma _{\dagger},\alpha ),(0,1))\ast (\Sigma _{1,1},(1,0))$). 
\vspace{1mm}

B) Suppose that $d=10$.
\vspace{1mm}

B1) Consider any compatible couple
$\mathcal{K}_{\flat}:=(\sigma _{\flat},(pos, neg))$
with $2\leq pos \leq neg$. Represent the sign pattern $\sigma _{\flat}$
as above in the form
$(\sigma _{\dagger},\alpha ,\beta )$.

If $\alpha \neq \beta$ and the couple
$\mathcal{K}_{\triangle}:=((\sigma _{\dagger},\alpha ), (pos-1, neg))$ is realizable,
then the couple $\mathcal{K}_{\flat}$ is also realizable as
$\mathcal{K}_{\triangle}\ast (\Sigma _{1,1},(1,0))$.

If $\alpha =\beta$ and the couple
$\mathcal{K}_{\diamond}:=((\sigma _{\dagger},\alpha ), (pos, neg-1))$ is realizable,
then one can realize $\mathcal{K}_{\flat}$ as 
$\mathcal{K}_{\diamond}\ast (\Sigma _{2},(0,1))$. However the couple
$\mathcal{K}_{\diamond}$
is realizable. Indeed, it is with $d=9$ and either with $\tilde{\mu}=1$ when
$pos=neg=2$ (and there
exist no such non-realizable couples) or with $\tilde{\mu}\geq 2$
when $neg\geq 3$,
hence again realizable. So the only situation in which
one does not know whether the couple $\mathcal{K}_{\flat}$ is realizable or not
is when

$$\alpha \neq \beta ~~~\, {\rm and}~~~\,
\mathcal{K}_{\triangle}=\mathcal{K}_{4,4}=(\Sigma _{1,4,4,1},(1,6))~.$$
This means that $\mathcal{K}_{\flat}=(\Sigma _{1,4,4,1,1},(2,6))$
which is realizable as 
$(\Sigma _{1,4,4},(2,6))\ast (\Sigma _{1,1,1},(0,0))$. Hence all
couples with $d=10$ and $\tilde{\mu}\geq 2$ are realizable.
\vspace{1mm}

B2) We need some more notation:

\begin{nota}\label{notaK}
  {\rm We denote by $\mathcal{K}_{\bullet}$ a couple with $\tilde{\mu}=1$ and by
    $\sigma _{\bullet}:=(\alpha _d,\ldots ,\alpha _0)$ its sign pattern, where
    $\alpha _j=+$ or
    $-$, $\alpha _d=+$. We set
    $\sigma ^s:=\alpha _s(\alpha _s,\ldots ,\alpha _0)$ and
    $\sigma _s:=(\alpha _d,\ldots ,\alpha _s)$. We discuss the possibility
    to realize $\mathcal{K}_{\bullet}$ as $\mathcal{K}_s\ast \mathcal{K}^s$,
    where the couples $\mathcal{K}_s$, $\mathcal{K}^s$ correspond to degree
    $d-s$ and $s$ polynomials with sign patterns $\sigma _s$ and $\sigma ^s$.}
  \end{nota}

\begin{rem}\label{remK}
  {\rm If $\alpha _s=+$, then the admissible pair of $\mathcal{K}_s$ (resp.
    of $\mathcal{K}^s$) is of the form $(0,.)$ (resp. $(1,.)$). If
    $\alpha _s=-$, then these admissible pairs are
    of the form $(1,.)$ and $(0,.)$ respectively. We remind that for
    $d\leq 8$, the admissible pairs of the form $(1,.)$ are realizable, see the
  beginning of Section~\ref{secprelim} and part A) of this proof.}
\end{rem}

Suppose first that $s=5$ and $\alpha _5=+$.
Then the couple $\mathcal{K}^5$ is realizable.
If $\mathcal{K}_5$ is realizable, then such is $\mathcal{K}_{\bullet}$ as well.
The only
possibility $\mathcal{K}_5$ not to be realizable is when
$\mathcal{K}_5=(\Sigma _{1,4,1},(0,3))$, see (\ref{eqd4d5}). So we assume that 
$\sigma _{\bullet}=(+,-,-,-,-,+,?,?,?,?,-)$. If $\alpha _4=-$,
then the couple $\mathcal{K}_4$
is realizable and $\mathcal{K}^4$ is not realizable only when
$\mathcal{K}^4=(\Sigma _{1,3,1},(0,2))$. This means that
$\sigma _{\bullet}=\Sigma _{1,4,1,1,3,1}$, with Descartes' pair $(5,5)$. This
sign pattern is realizable with the admissible pairs $(1,a)$, $a=3$ or~$5$:

\begin{equation}\label{eqrealiz}
  (\Sigma _{1,4,1,1,3,1},(1,a))=((\Sigma _{1,3},(1,a-3))~\ast ~
  (\Sigma _{2,1,1,3,1},(0,3))~.
  \end{equation}
If $\alpha _5=\alpha _4=+$, then $\mathcal{K}^4$ is realizable whereas
the couple
$\mathcal{K}_4$ is not realizable only when it corresponds to one of the
four cases
listed in~(\ref{eqd=6}). In each of them one has to take into account the
involution $i_r$ (see (\ref{eqir})), but not the involution $i_m$,
because the latter makes the first component of the admissible pair
larger than~$1$. Hence $\sigma _4=\Sigma _{1,4,2}$. Then we consider
the couples $\mathcal{K}_3$ and $\mathcal{K}^3$. The latter
is always realizable, so we assume that this is not the case of $\mathcal{K}_3$.
Hence $\sigma _3$ is obtained from $\sigma _4$ by adding a sign $+$
to the right, i.e. $\sigma _3=\Sigma _{1,4,3}$.

If $\sigma _2=\Sigma _{1,4,3,1}$, then both $\mathcal{K}_2$ and $\mathcal{K}^2$
are realizable, so we need to consider only the possibility
$\sigma _2=\Sigma _{1,4,4}$. Thus in the end
$\sigma _{\bullet}=\Sigma _{1,4,5,1}$ or $\Sigma _{1,4,4,2}$. 
These two sign patterns are realizable with the admissible pairs $(1,a)$,
$a=3$ and~$5$:

$$\begin{array}{l}
  (\Sigma _{1,4,5,1},(1,a))=(\Sigma _{1,4,1},(0,1))~\ast ~
(\Sigma _{5,1},(1,a-1))~,\\ \\ 
(\Sigma _{1,4,4,2},(1,a))=(\Sigma _{1,4,1},(0,1))~\ast ~
(\Sigma _{4,2},(1,a-1))~.\end{array}$$
The first of them is not realizable with the admissible pair $(1,7)$
(see Theorem~\ref{tmmain}), for the
second one the answer is not known.

Suppose now that $\alpha _5=-$. Then $\mathcal{K}_5$ is realizable while
$\mathcal{K}^5$ might not be only if $\sigma ^5=\Sigma _{1,4,1}$. In this case 
we consider $\mathcal{K}_6$ and $\mathcal{K}^6$. For $\alpha _6=+$,
the couple $\mathcal{K}^6$ is realizable whereas $\mathcal{K}_6$ could be
non-realizable, if $\sigma _6=\Sigma _{1,3,1}$ hence
$\sigma _{\bullet}=\Sigma _{1,3,1,1,4,1}$. This is the case (\ref{eqrealiz})
to which
one has applied the involution $i_r$, so it is realizable.
For $\alpha _5=\alpha _6=-$, one is, up to the involution $i_r$, in the case
$\alpha _{10}=\alpha _5=\alpha _4=+$, $\alpha _0=-$, which was already
considered.
 \vspace{1mm}

  C) Suppose that $d=11$.
  \vspace{1mm}
  
  C1) We use the same notation as the one used in part B1)
  of this proof; in particular, we denote by
  $\mathcal{K}_{\flat}:=(\sigma _{\flat},(pos, neg))$ a compatible couple with
  $\tilde{\mu}\geq 2$, where
  $\sigma _{\flat}=(\sigma _{\dagger},\alpha ,\beta )$.
  As in part B1) of this proof we show that the only cases in which the couple
  $\mathcal{K}_{\flat}$ is possibly non-realizable are when $\alpha \neq \beta$
  and 

  $$\mathcal{K}_{\triangle}=((\sigma _{\dagger},\alpha ),(pos-1,neg))=
  (\Sigma _{1,4,5,1},(1,7))=\mathcal{K}_{4,5}~~~\, {\rm or}~~~\,
  (\Sigma _{1,4,4,2},(1,7))~,$$
  that is, $\mathcal{K}_{\triangle}$ corresponds to one of the two cases
  mentioned in part (2) of the present theorem. 
  In these cases $\mathcal{K}_{\flat}$ is realizable, because it
  equals respectively

  $$\begin{array}{cclc}
    (\Sigma _{1,4,5,1,1},(2,7))&=&(\Sigma _{1,4,1},(2,3))
  \ast (\Sigma _{5,1,1},(0,4))&
       {\rm and}\\ \\ (\Sigma _{1,4,4,2,1},(2,7))&=&
       (\Sigma _{1,4,1},(2,3))\ast (\Sigma _{4,2,1},(0,4))~.&\end{array}$$

  C2) We use the notation and method of proof as developed in part B2) of
  this proof. We are looking for non-realizable couples $\mathcal{K}_{\bullet}$.
  Suppose first that
  \vspace{1mm}
  
  \noindent 1. $\alpha _5=+$. Then the couple
  $\mathcal{K}^5$ is realizable and the couple $\mathcal{K}_5$ might not be
  realizable only if $\sigma _5$ corresponds to one of the cases (\ref{eqd=6})
  up to the involution $i_r$. This means that $\sigma _5$ is among the sign
  patterns

  $$\Sigma _{1,5,1}~,~~~\, \Sigma _{1,1,1,3,1}~,~~~\, \Sigma _{1,3,1,1,1}~,~~~\,
  \Sigma _{2,4,1}~~~\, {\rm and}~~~\, \Sigma _{1,4,2}~.$$
  1.1. If $\alpha _4=+$, then $\mathcal{K}^4$ is realizable and
  $\sigma _4$ is one of the sign patterns

  $$\Sigma _{1,5,2}~,~~~\, \Sigma _{1,1,1,3,2}~,~~~\, \Sigma _{1,3,1,1,2}~,~~~\,
  \Sigma _{2,4,2}~~~\, {\rm and}~~~\, \Sigma _{1,4,3}~.$$
  1.1.1. The second and third sign patterns of this list do not correspond to
  non-realizable cases, see (\ref{eqd=7}).
  \vspace{1mm}
  
  \noindent 1.1.2. In the cases
  $\sigma _4=\Sigma _{1,5,2}$ and $\sigma _4=\Sigma _{2,4,2}$ the first component
  of the sign pattern $\sigma ^8$ must be $3$ and the second must be $\geq 2$;
  the sign pattern $\sigma _8$ corresponds to realizable couples. Then one
  concludes from (\ref{eqd=8}) that $\sigma _{\bullet}$ is one of the
  sign patterns

  $$\Sigma _{1,5,5,1}~,~~~\, \Sigma _{1,5,4,2}~,~~~\, \Sigma _{2,4,5,1}~~~\,
  {\rm and}~~~\, \Sigma _{2,4,4,2}~.$$
  The first of them is not realizable with the admissible pair $(1,8)$, see
  Theorem~\ref{tmmain};
  for the second, third and fourth one the answer to this question
  is not known. (The second and third are in one and the same orbit.) We show
  the realizability of the four sign patterns with the other admissible pairs
  $(1,a)$, $a=6$,
  $4$ or~$2$:

  $$\begin{array}{l}
    (\Sigma _{1,5,5,1},(1,a))=(\Sigma _{1,5,3},(0,a-2))~\ast ~
    (\Sigma _{3,1},(1,2))\\ \\
    (\Sigma _{1,5,4,2},(1,a))=(\Sigma _{1,5,3},(0,a-2))~\ast ~
    (\Sigma _{2,2},(1,2))\\ \\
    (\Sigma _{2,4,4,2},(1,a))=(\Sigma _{2,4,3},(0,a-2))~\ast ~
    (\Sigma _{2,2},(1,2))~.
\end{array}$$
  
  \noindent 1.1.3. In the case $\sigma _4=\Sigma _{1,4,3}$ the first
  component of $\sigma ^8$ equals $2$ and the second is $\geq 3$.
  Using the list (\ref{eqd=8}) one obtains the following possibilities
  for~$\sigma _{\bullet}$:

  $$\Sigma _{1,4,5,2}~,~~~\, \Sigma _{1,4,6,1}~,~~~\,
  \Sigma _{1,4,4,3}~~~\, {\rm and}~~~\, \Sigma _{1,4,4,1,1,1}~.$$
  The second of them is not realizable with the admissible pair
  $(1,8)$ (see Theorem~\ref{tmmain}). The third sign pattern is realizable
  with the admissible pair $(1,8)$. Indeed, by
  \cite[Theorem~3, part (iii)]{CGK1}, the couple
  $(\Sigma _{4,4,3},(0,8))$ is realizable, so one can set

$$(\Sigma _{1,4,4,3},(1,8))=((+,-),(1,0))\ast (\Sigma _{4,4,3},(0,8))~.$$ 
  The realizability of other possible cases
  is given below; for the cases which are not covered by this list
  the answer is not known.

  $$\begin{array}{ll}
(\Sigma _{1,4,5,2},(1,a))=(\Sigma _{1,4,4},(0,a-2))~\ast ~
    (\Sigma _{2,2},(1,2))~,&a=6,~4~~~\, {\rm or}~~~\, 2~;\\ \\
    (\Sigma _{1,4,6,1},(1,a))=(\Sigma _{1,4,4},(0,a-2))~\ast ~
    (\Sigma _{3,1},(1,2))~,&a=6,~4~~~\, {\rm or}~~~\, 2~;\\ \\
    (\Sigma _{1,4,4,3},(1,a))=(\Sigma _{1,4,4},(0,a-2))~\ast ~
    (\Sigma _{1,3},(1,2))~,&a=6,~4~~~\, {\rm or}~~~\, 2~;\\ \\
    (\Sigma _{1,4,4,1,1,1},(1,a))=(\Sigma _{1,4,4},(0,a-2))~\ast ~
    (\Sigma _{1,1,1,1},(1,0))~,&a=4~~~\, {\rm or}~~~\, 2~.
    \end{array}$$
  
  \noindent 1.2. If now $\alpha _4=-$, the couple $\mathcal{K}_4$ is
  realizable, so one
  has to treat only the situation $\sigma ^4=\Sigma _{1,3,1}$ in which
  $\mathcal{K}^4$ is not realizable with the admissible pair $(0,2)$. This,
  combined with the list (\ref{eqd=6}), 
  gives the following sign patterns $\sigma _{\bullet}$:

  $$\Sigma _{1,5,1,1,3,1}~,~~~\, \Sigma _{1,1,1,3,1,1,3,1}~,~~~\,
  \Sigma _{1,3,1,1,1,1,3,1}~,~~~\,
  \Sigma _{2,4,1,1,3,1}~~~\, {\rm and}~~~\, \Sigma _{1,4,2,1,3,1}~.$$

  \noindent 1.2.1. The first, fourth and fifth sign patterns
  are realizable with the admissible pairs $(1,6)$, $(1,4)$ and $(1,2)$:

  $$\begin{array}{l}
    (\Sigma _{1,5,1,1,3,1},(1,a+b))=(\Sigma _{1,4},(1,a))\ast 
    (\Sigma _{2,1,1,3,1},(0,b))~,~~~a,b=1~~~{\rm or}~~~3~,\\ \\
    (\Sigma _{2,4,1,1,3,1},(1,a+b))=(\Sigma _{2,3},(1,a))\ast
    \Sigma _{2,1,1,3,1},(0,b))~,~~~a,b=1~~~
           {\rm or}~~~3~,\\ \\
           (\Sigma _{1,4,2,1,3,1},(1,a+b))=(\Sigma _{1,4},(1,a))\ast 
           (\Sigma _{1,2,1,3,1},(0,b))~,~~~a,b=1~~~
           {\rm or}~~~3~.
    \end{array}$$

  \noindent 1.2.2. The second (resp. the third) sign patterns is
  realizable with
  the admissible pairs $(1,4)$ and $(1,2)$ (resp. $(1,2)$);
  for the third sign pattern and
  for the admissible pair $(1,4)$, the answer is not known:

  $$\begin{array}{l}
    (\Sigma _{1,1,1,3,1,1,3,1},(1,4))=(\Sigma _{1,1,1},(0,0))~\ast ~
    (\Sigma _{1,3,1,1,3,1},(1,4))~,\\ \\
    (\Sigma _{1,1,1,3,1,1,3,1},(1,2))=(\Sigma _{1,1,1,3,1,1,1},(0,0))~\ast ~
    (\Sigma _{3,1},(1,2))~,\\ \\
    (\Sigma _{1,3,1,1,1,1,3,1},(1,2))=(\Sigma _{1,3,1,1,1,1,1},(0,0))~\ast ~
    (\Sigma _{3,1},(1,2))~.
  \end{array}$$

  \noindent 2. $\alpha _5=-$. The couple $\mathcal{K}_5$ is realizable while
  $\mathcal{K}^5$ is not only if $\sigma ^5=\Sigma _{1,4,1}$. 

  \noindent 2.1. If $\sigma ^5=\Sigma _{1,4,1}$ and $\alpha _6=+$, then
  $\mathcal{K}^6$ is realizable whereas $\mathcal{K}_6$ is not realizable
  only if $\sigma _6=\Sigma _{1,4,1}$, so $\sigma _{\bullet}=\Sigma _{1,4,1,1,4,1}$.
  This sign pattern is realizable with the admissible pairs $(1,4)$ and
  $(1,2)$ (for $(1,6)$, the answer remains unknown):

  $$
  (\Sigma _{1,4,1,1,4,1},(1,4)~~~\, {\rm or}~~~\, (1,2))=
  (\Sigma _{1,3},(1,2))~\ast ~
    (\Sigma _{2,1,1,4,1},(0,2)~~~\, {\rm or}~~~\, (0,0))~.$$

  \noindent 2.2. If $\alpha _5=-$, $\sigma ^5=\Sigma _{1,4,1}$ and
  $\alpha _6=-$, i.~e.
  $\sigma ^6=\Sigma _{2,4,1}$, then applying the involution $i_r$
  (see (\ref{eqir})) one
  transforms this case into $\sigma _5=\Sigma _{1,4,2}$, $\alpha _0=-$.
  This case was treated in~1.
\vspace{1mm}

D) Suppose that $d=12$. One can try to represent a given couple
  $\mathcal{K}_{\bullet}$ with $\tilde{\mu}\geq 2$
  as a concatenation of couples $\mathcal{K}'$
  and $\mathcal{K}''$ of degree $10$ and $2$ respectively. If $\mathcal{K}'$ is
  realizable, then such is $\mathcal{K}_{\bullet}$ as well, because all
  compatible couples of degree $2$ are realizable. So we assume that 
  $\mathcal{K}'$ is not realizable. Then one can assume that $\mathcal{K}'$
  is one of the two couples of part~(2) of the theorem. Hence the sign pattern
  of $\mathcal{K}'$ is among the following ones:

  $$\begin{array}{llll}
    \Sigma _{1,4,5,3}~,&\Sigma _{1,4,5,2,1}~,&\Sigma _{1,4,5,1,2}~,&
    \Sigma _{1,4,5,1,1,1}~,\\ \\
    \Sigma _{1,4,4,4}~,&\Sigma _{1,4,4,3,1}~,&
    \Sigma _{1,4,4,2,2}~,&\Sigma _{1,4,4,2,1,1}~.\end{array}$$
  Represent these sequences of $4$, $5$ or $6$ numbers in the forms
  $(1,A)$ and $(1,4,B)$, where $A$ is the sequence of the last
  $3$, $4$ or $5$
  of them, and $B$ of the last $2$, $3$ or $4$ respectively.
  If $3\leq pos \leq neg$, then

  \begin{equation}\label{eq12}
    \mathcal{K}_{\bullet}=(\Sigma _{1,1},(1,0))~\ast ~\mathcal{K}^{\sharp}~,~~~\,
            {\rm where}~~~\, \mathcal{K}^{\sharp}=(\Sigma _{1,A},(pos-1,neg))~.
            \end{equation}
  The couple $\mathcal{K}^{\sharp}$ is with $d=11$ and $\tilde{\mu}\geq 2$ hence
  realizable. Hence $\mathcal{K}_{\bullet}$ is also realizable.

  Suppose that $2=pos\leq neg$. Then one can write

  \begin{equation}\label{eq12bis}
    \mathcal{K}_{\bullet}=(\Sigma _{1,4},(1,3)~~~\, {\rm or}~~~\, (1,1))~\ast ~
    (\Sigma _{1,B},(a,b))~,
    \end{equation}
  where the admissible pair of $\mathcal{K}_{\bullet}$ is $(a+1,b+3)$ if
  $neg\geq 4$ or $(a+1,b+1)$ if $neg=2$ or~$3$. Each couple
  $(\Sigma _{1,4},(1,3))$, $(\Sigma _{1,4},(1,1))$ and $(\Sigma _{1,B},(a,b))$ is
  realizable (for the latter one has $d=8$ and $\tilde{\mu}\geq 1$), so
  $\mathcal{K}_{\bullet}$ is realizable.
  \vspace{1mm}

 E) Suppose that $d=13$ or $d=14$. (For $d=14$, we assume that the theorem is
  proved for $d=13$.) Similarly to part D), one can try to represent
  a given couple
  $\mathcal{K}_{\bullet}$ with $\tilde{\mu}\geq 2$
  as a concatenation of couples $\mathcal{K}'$
  and $\mathcal{K}''$ of degree $10$ and $3$ respectively (or $10$ and $4$
  if $d=14$). As in part D),
  we assume that $\mathcal{K}'$ is one of the couples of part~(2)
  of the theorem. We denote the sign pattern of $\mathcal{K}_{\bullet}$
  by $(1,A)$ and $(1,4,B)$ with similar definition of $A$ and $B$. 

  For $3\leq pos \leq neg$, we use formula (\ref{eq12}). The couple
  $\mathcal{K}^{\sharp}$ is with $d=12$ (resp. $d=13$)
  and $\tilde{\mu}\geq 2$ hence realizable. For
  $2=pos\leq neg$, we can use formula (\ref{eq12bis}) with the reasoning after
  it except in the case when $(\Sigma _{1,B},(a,b))$ is of the orbit of
  the couple (\ref{eqcased9}) (resp. of one of the orbits of the couples
  of part~(2) of the theorem). The first $8$ signs of the sign pattern
  of $\mathcal{K}_{\bullet}$ are $(+,-,-,-,-,+,+,+)$. Further we treat
  separately the cases $d=13$ and $d=14$. 

Suppose that $d=13$. These first $8$ signs imply that from the orbit
  of the sign pattern $\Sigma _{1,4,4,1}=i_r(\Sigma _{1,4,4,1})$ one has to
  choose $\Sigma _{1,4,4,1}$ (and not
  $\Sigma _{2,1,1,2,1,1,2}=i_m(\Sigma _{1,4,4,1})=i_r(\Sigma _{2,1,1,2,1,1,2})$)
  to be equal to $\Sigma _{1,B}$. Thus the sign pattern of
  $\mathcal{K}_{\bullet}$ is $\Sigma _{1,4,4,4,1}$. The admissible pair equals
  $(2,\nu )$, $\nu =9$, $7$, $5$ or $3$. We set

  $$(\Sigma _{1,4,4,4,1},(2,\nu ))=(\Sigma _{1,2},(1,1))~\ast ~
  \mathcal{K}^*~\ast ~(\Sigma _{1,1},(1,0))~,~~~\,
  \mathcal{K}^*:=(\Sigma _{3,4,4},(0,\nu -1))$$
  The first and the third couple in this concatenation are realizable.
  The couple $\mathcal{K}^*$ is realizable for $\nu =9$,
  see \cite[Theorem~4]{CGK1}.
  Hence there exists a real polynomial $Q_{\natural}$
  with sign pattern $\Sigma _{3,4,4}$
  and having $8$ negative distinct roots and a complex conjugate pair.
  One can perturb the coefficients of
  $Q_{\natural}$ without changing its sign pattern so that all its critical
  values become distinct.
  Hence for suitable positive values of $t$ one obtains polynomials
  $Q_{\natural}+t$ having the sign pattern $\Sigma _{3,4,4}$ and with
  exactly $6$, $4$ or $2$ distinct negative roots and, respectively,
  $2$, $3$ or $4$ conjugate pairs. Thus the couple $\mathcal{K}^*$
  (and $(\Sigma _{1,4,4,4,1},(2,\nu ))$ as well)
  is realizable for $\nu =9$, $7$, $5$ or $3$.
  \vspace{1mm}

F) Suppose that $d=14$. The first $8$ signs of the sign pattern
  of $\mathcal{K}_{\bullet}$ (see E)) mean that the sign pattern
  $\Sigma _{1,B}$ is in one of the orbits of the sign patterns
  $\Sigma _{1,4,5,1}$
  or $\Sigma _{1,4,4,2}$, see part~(2) of the theorem; the first component of
  $B$
  must be $\geq 3$. Hence $\Sigma _{1,B}$ is among the following sign patterns:

  $$\Sigma _{1,4,5,1}~,~~~\, \Sigma _{1,5,4,1}~~~\, {\rm or}~~~\,
  \Sigma _{1,4,4,2}~.$$
  For $\Sigma _{1,B}=\Sigma _{1,4,4,2}$, one has

  $$\mathcal{K}_{\bullet}=(\Sigma _{1,4,4,4,2},(2,\nu +1))=
  (\Sigma _{1,4,4,4,1},(2,\nu ))~\ast ~(\Sigma _{2},(0,1))~,$$
  where $\nu =9$, $7$, $5$ or $3$ and the first concatenation factor is
  realizable (see E)). Hence $\mathcal{K}_{\bullet}$ is also realizable.
  For $\Sigma _{1,B}=\Sigma _{1,4,5,1}$, one has

  $$\mathcal{K}_{\bullet}=(\Sigma _{1,4,4,5,1},(2,\nu +1))=(\Sigma _{1,2},(1,1))
  ~\ast ~\mathcal{K}^*~\ast ~(\Sigma _{2,1},(1,1))$$
  with $\mathcal{K}^*$ as in~E), so this case is also realizable. For
  $\Sigma _{1,B}=\Sigma _{1,4,5,1}$, the sign pattern of $\mathcal{K}_{\bullet}$
  equals $\Sigma _{1,4,5,4,1}$. It is realizable with the admissible pairs
  $(2,\rho )$, $\rho =2$, $4$, $6$ or $8$:

  $$(\Sigma _{1,4,5,4,1},(2,\rho ))=(\Sigma _{1,4,1},(0,1))~\ast ~
  (\Sigma _{5,4,1},(2,\rho -1))~.$$
  For $\rho =10$, the answer to Problem~\ref{prob1} remains unknown.

\section{Proof of Theorem~\protect\ref{tm2SC}\protect\label{secprtm2SC}}

A) We consider first the (non)-realizable cases with the admissible pairs
$(0,1)$ and $(0,3)$. Every couple $(\Sigma _{m,n,q},(0,1))$
with $m+n+q=2\ell$, $\ell \geq 2$, $m>0$, $n>0$, $q>0$, is realizable --
it suffices to choose
a polynomial with sign pattern $\Sigma _{m,n,q}$ and to add to it a large
positive constant. Further we suppose that $q\geq m$,
otherwise one considers the couple from
the same orbit $(\Sigma _{q,n,m},(0,3))$ using the involution~$i_r$.

On the other hand any couple of the form
$(\Sigma _{d+1},(0,k))$ is realizable (\cite[Example~1]{CGK1}). This implies
that for $d=9$, any couple with $q\geq 3$ and admissible pair $(0,3)$ is
realizable as

$$(\Sigma _{m,n,q},(0,3))=(\Sigma _{m,n,q-2},(0,1))\ast (\Sigma _3,(0,2))~.$$
The couple $(\Sigma _{2,6,2},(0,3))$ is realizable as

$$(\Sigma _{2,6,2},(0,3))=(\Sigma _{2,6,1},(0,2))\ast (\Sigma _2,(0,1))~,$$
where the first factor is from the same orbit as

$$(\Sigma _{1,6,2},(0,2))=(\Sigma _{1,6,1},(0,1))\ast (\Sigma _2,(0,1))$$
hence this orbit is realizable. We prove realizability of the orbit
$(\Sigma _{1,7,2},(0,3))$ by direct construction. We set

$$G_0:=x(x^2-1)^2(x^4+2x^2+1)=x^9-2x^5+x~,$$
and then $G_1:=G_0-0.0001(x+1)^8+0.2$ which has three negative roots
  $-1.09\ldots$, $-0.84\ldots$ and $-0.20\ldots$ and three complex conjugate
pairs of roots. The sign pattern defined by $G_1$ is $\Sigma _{1,7,2}$,
because

  $$\begin{array}{ccl}
    G_1&=&x^9-0.0001x^8-0.0008x^7-0.0028x^6-2.0056x^5-0.0070x^4\\ \\
    &&-0.0056x^3-0.0028x^2+
  0.9992x+0.1999~.\end{array}$$
Finally, the orbit $(\Sigma _{1,8,1},(0,3))$ is not realizable, see
\cite[part (i) of Theorem~4]{KoSh1} (one has to apply the involution $i_m$
to the series of non-realizable examples described there). 
\vspace{1mm}

B) Suppose that the admissible pair is $(0,5)$. If $q\geq 5$, then one
realizes the couple by the concatenation

$$(\Sigma _{m,n,q-4},(0,1))\ast (\Sigma _5,(0,4))~.$$
Suppose that either $q=4$ or $q=3$ and $m=2$ or $3$. Then by
\cite[Theorem~9]{FoKoSh} the corresponding couple
$(\Sigma _{m,n,q-2},(0,3))$ is realizable
and one sets

$$(\Sigma _{m,n,q},(0,5))=(\Sigma _{m,n,q-2},(0,3))\ast (\Sigma _3,(0,2))~.$$
The couple $(\Sigma _{1,8,1},(0,5))$ is not realizable (see \cite[part (i) of Theorem~4]{KoSh1}). The following
proposition settles the remaining cases:

\begin{prop}\label{prop3cases}
  The couples $(\Sigma _{1,7,2},(0,5))$, $(\Sigma _{1,6,3},(0,5))$ and
  $(\Sigma _{2,6,2},(0,5))$ are not realizable.
\end{prop}
The proposition is proved in Section~\ref{secprprop3cases}.
\vspace{1mm}

C) Suppose that $d=9$ and the admissible pair is $(0,7)$. The couple $(\Sigma _{1,8,1},(0,7))$ is not realizable, see \cite[part (i) of Theorem~4]{KoSh1}. It follows from 
\cite[Theorem~3]{CGK1} that if $n=1$, $2$ or $3$, then such a sign pattern is
realizable with the admissible pair $(0,7)$. The same theorem implies that
for $m=1$, $n\geq 4$, the couple $(\Sigma _{m,n,q},(0,7))$ is not realizable.
Proposition~1 in \cite{CGK1} says that the couples
$(\Sigma _{3,4,3},(0,7))$ and $(\Sigma _{2,4,4},(0,7))$ are not realizable.
The remaining couples $(\Sigma _{2,5,3},(0,7))$ and $(\Sigma _{2,6,2},(0,7))$
are not realizable by \cite[Proposition~6]{FoKoSh}.

\section{Proof of Proposition~\protect\ref{prop3cases}
  \protect\label{secprprop3cases}}

\begin{defi}
  {\rm (1) A {\em generalized sign pattern} is a string of signs $+$, $-$
    and/or $0$
    beginning with a $+$. A real polynomial $P$ with positive leading
    coefficient is said to define a given
    generalized sign pattern $\sigma$ if the components of $\sigma$ are equal 
    to the signs of the corresponding coefficients of $P$. Given a sign
    pattern $\sigma _1$ and a generalized sign pattern $\sigma _2$
    of the same length, $\sigma _2$ is called {\em adjacent} to $\sigma _1$ if
    it is obtained from $\sigma _1$ by replacing some of its components
    (excluding the initial $+$) by zeros. The {\em closure} of a given
    sign pattern is the set containing the sign pattern and all generalized
    sign patterns adjacent to it. 

    (2) We call {\em simultaneous shift} a map
    $\tau _{\pm}:x\mapsto x\pm \varepsilon$, where $\varepsilon >0$. We usually
    consider $\varepsilon$ to be sufficiently small.}
\end{defi}

\begin{lm}\label{lmshift}
  Given a real monic polynomial $W:=\sum _{j=0}^da_jx^j$
  with at least one vanishing coefficient,
  there exists a simultaneous shift $\tau _{\pm}$
  after which all coefficients of $W$ are
  non-zero and the initially non-zero coefficients keep their signs.
  Moreover, if $a_k=0\neq a_{k+1}$, then one can choose the sign $+$ or $-$
  in the definition of $\tau _{\pm}$ so that after the shift the sign of
  the new coefficient $a_k$ be the desired one ($+$ or $-$).
\end{lm}

\begin{proof}
  After the shift all coefficients of $W$ become non-constant polynomials in
  $\varepsilon$, so with the exception of finitely-many values of
  $\varepsilon$ they are all non-zero. After the shift the coefficient
  $a_k$ becomes $a_k\pm (k+1)a_{k+1}\varepsilon +o(\varepsilon )$
  from which the last statement of the lemma follows.
  \end{proof}

\begin{lm}\label{lmnotexist7}
  For $d=7$, there exists no monic polynomial $P:=\sum _{j=0}^7a_jx^j$
  satisfying simultaneously the following conditions:

  (i) one has $P(0)>0$ and $P$ defines either one of the sign patterns
  $\Sigma _{1,6,1}$ or
  $\Sigma _{1,5,2}$ or a generalized sign pattern adjacent to one of them;

   (ii) $P$ has $5$ negative roots counted with multiplicity;

  (iii) $P$ has either a complex conjugate pair of roots or a double
  positive root.
\end{lm}

\begin{proof}
  Suppose first that $P$ has a complex conjugate pair of roots. Denote by
  $-\eta _j$ the negative roots of $P$. Set
  $P_1:=x(x+\eta _1)\cdots (x+\eta _5)$. Hence all coefficients of $P_1$
  are positive. For $\varepsilon >0$ small enough, the polynomial
  $P_2:=P-\varepsilon P_1$ defines one of the sign patterns $\Sigma _{1,6,1}$ or
  $\Sigma _{1,5,2}$ and has a complex conjugate pair. One can perturb the
  negative roots of $P_2$ to make them all distinct while keeping the signs of
  its coefficients the same and the presence of a complex conjugate pair.
  However such a polynomial $P_2$ does not exist, see \cite[Theorem~9]{FoKoSh}.

  Suppose that $P$ has a double positive root. Hence $P$ is hyperbolic. If
  $P$ has no vanishing coefficients, then one can perturb the negative roots
  of $P$ to make them all distinct without changing the signs of the
  coefficients. After this one considers
  the polynomial $P_3:=P+\delta x^4$ with $5$ distinct negative roots and
  defining the same sign pattern as
  $P$ when $\delta >0$ is small enough. The polynomial $P_3$ has no
  double positive root, but a complex conjugate pair close to the double
  root of $P$. Again by \cite[Theorem~9]{FoKoSh} such a polynomial $P_3$
  does not exist.

  Suppose that $P$ has a double positive root and at least one vanishing
  coefficient. We remind that $P$ cannot have two consecutive vanishing
  coefficients (\cite[Lemma 7]{KoMB}). For $\varepsilon >0$ small enough,
  all coefficients of $P$ can be made non-zero as a result of a shift
  $\tau _{\pm}$. We consider the following cases:
  \vspace{1mm}

   1) One has $a_6<0$ and $a_1<0$ (resp. $a_2<0$). A shift $\tau_+$ or
  $\tau _-$ makes
  all coefficients between $a_6$ and $a_1$ (resp. between $a_6$ and $a_2$)
  non-zero. Hence they are all negative, otherwise by Descartes' rule of signs
  $P$ cannot have $5$ negative roots. One perturbs the negative roots of $P$
  to make them distinct while keeping the sign pattern. Then for $\delta >0$
  small enough, the polynomial $P+\delta x^4$ has still $5$ distinct negative
  roots and the same sign pattern, but the double root gives birth to a
  complex conjugate pair close to it. Such a polynomial does not exist,
  see \cite[Theorem~9]{FoKoSh}.
  \vspace{1mm}

  2) One has $a_6=0$ and $a_1<0$ (resp. $a_2<0$). A shift $\tau _{\pm}$ with
  suitably chosen sign $+$ or $-$ makes all coefficients between $a_7$ and
  $a_1$
  (resp. between $a_7$ and $a_2$) non-zero, and in particular $a_6$
  becomes negative. Again by Descartes' rule of signs the rest of
  the coefficients between $a_7$ and $a_1$ (resp. between $a_7$ and $a_2$) are
  negative and as in case 1) one concludes that such a polynomial does not
  exist.
  \vspace{1mm}

   3) One has $a_6<0$ and $a_1=0$ (resp. $a_2=0$). A shift $\tau _{\pm}$ with
   suitably chosen sign $+$ or $-$ makes all coefficients between $a_6$ and
   $a_0$
  (resp. between $a_6$ and $a_1$) non-zero, and in particular $a_1$ (resp.
  $a_2$) becomes negative. The rest of the reasoning is as in case 2).
  \vspace{1mm}

   4) One has $a_6=a_1=0$ (resp. $a_6=a_2=0$). Then $a_2\neq 0$ (resp.
  $a_3\neq 0$), see \cite[Lemma 7]{KoMB}, hence $a_2<0$ (resp. $a_3<0$).
  A shift $\tau _{\pm}$ with
  suitably chosen sign $+$ or $-$ makes all coefficients between $a_7$ and
  $a_0$
  (resp. between $a_7$ and $a_1$) non-zero, and in particular $a_6$
  becomes negative. Hence the sign pattern of $P$ is now $\Sigma _{1,6,1}$,
  $\Sigma _{1,5,2}$ or $\Sigma _{1,4,3}$. One perturbs the negative roots of $P$
  to make them distinct while preserving the sign pattern, and then adds to
  $P$ the monomial $\delta x^4$ with $\delta >0$ small enough. The double root
  gives birth to a complex conjugate pair and $P$ realizes one of the couples
  $(\Sigma _{1,k,7-k},(0,5))$, $k=4$, $5$ or~$6$, which by
  \cite[Theorem~9]{FoKoSh} is impossible.
\end{proof}

\begin{lm}\label{lmnotexist9}
    There exists no real monic degree $9$ polynomial $U$ satisfying the
    following conditions:

    (i) its sign pattern is $\Sigma _{1,7,2}$, $\Sigma _{1,6,3}$ or
    $\Sigma _{2,6,2}$ or it is a
    generalized sign pattern adjacent to one of these sign patterns;

    (ii) it has $7$ negative roots counted with multiplicity;

    (iii) it has either a double positive root or a complex conjugate pair.
\end{lm}

 \begin{proof}
   First of all we observe that the last three components of the
   (generalized)
    sign pattern cannot be $(+,0,+)$, because in this case the polynomial
    $U(-x)$ has less than $7$ sign changes and by Descartes' rule of signs
    the polynomial cannot have $7$ negative roots. Hence the coefficients
    of $x^8$, $\ldots$, $x^3$ (or $x^7$, $\ldots$, $x^3$ in the case of
    $\Sigma _{2,6,2}$)
    are non-positive. The one of $x$ is
    non-negative and can be $0$ only if the one of $x^2$ is non-positive.

    If the polynomial $U$ has a double positive root, then it is hyperbolic.
    Recall that by \cite[Lemma 7]{KoMB} the polynomial $U$
    has no two consecutive
    vanishing coefficients. One
    can perform a simultaneous shift to make all coefficients non-zero and
    so that the coefficient of $x^8$ (or of $x^7$ in the case of
    $\Sigma _{2,6,2}$) is negative. By
    Descartes' rule of signs the sign pattern of the polynomial is now
    $\Sigma _{1,5,4}$, $\Sigma _{1,6,3}$, $\Sigma _{1,7,2}$, $\Sigma _{2,6,2}$ or
    $\Sigma _{2,7,1}$. One can perturb its
    negative roots to make them distinct and then add to
    the polynomial a monomial $\delta x^4$, where $\delta >0$ is small enough.
    Thus the sign pattern is preserved and the double positive root gives
    birth to a complex conjugate pair close to it. However such a polynomial
    does not exist, see \cite[Proposition~6]{FoKoSh}.

     If $U$ has a complex conjugate pair, then we change it to
     $U_1:=U-\varepsilon \prod _{j=1}^7(x+\eta _j)$, $-\eta _j$ being the
     negative roots of $U$. For $\varepsilon >0$
    small enough, all coefficients of $U_1$ with the possible exception only
    of the one of $x^8$ are non-zero. After a simultaneous shift the
    coefficient of $x^8$ becomes negative and the sign pattern now equals
    $\Sigma _{1,6,3}$, $\Sigma _{1,7,2}$, $\Sigma _{1,8,1}$ or $\Sigma _{2,6,2}$.
    Then one perturbs
    the negative roots to make them distinct, so the polynomial realizes
    one of the couples $(\Sigma _{1,6,3},(0,7))$, $(\Sigma _{1,7,2},(0,7))$,
    $(\Sigma _{1,8,1},(0,7))$ or $(\Sigma _{2,6,2},(0,7))$
    which is impossible, see \cite[Proposition~6]{FoKoSh}.
    
 \end{proof}

  \begin{defi}
    {\rm We call {\em multiplicity vector} the vector whose components
      are equal to the multiplicities of the negative roots of a real
      polynomial listed in the increasing order. The multiplicity vector
      $\vec{v}_1$ is {\em adjacent} to the multiplicity vector $\vec{v}_2$ if
      $\vec{v}_1$ is obtained by applying one or several times the operation
      of replacing two consecutive components by their sum.
      Example: $(5)$ is adjacent to $(3,2)$ which in turn is adjacent
      to $(3,1,1)$.}
  \end{defi}

  \begin{lm}\label{lmtransform}
    Suppose that there exists a degree $9$ real monic polynomial $V$ 
    realizing the couple $(\Sigma _{1,7,2},(0,5))$ (resp.
    $(\Sigma _{1,6,3},(0,5))$ or $(\Sigma _{2,6,2},(0,5))$).
    Then there exists a real monic degree $9$ polynomial satisfying the 
    following conditions:

    (i) it defines the sign pattern $\Sigma _{1,7,2}$ (resp. $\Sigma _{1,6,3}$
    or $\Sigma _{2,6,2}$);

    (ii) it has a double positive root; 
    
    (iii) the multiplicity vector of its negative roots is among the
    following ones:

     $$(1,2,2)~,~~~\, (2,1,2)~,~~~\, (3,1,1)~,~~~\, (1,1,3)~,~~~\, (3,2)~,
    ~~~\, (2,3)~.$$
  \end{lm}

  \begin{proof}
    The polynomial $V$ has either $4$ or $6$ critical points for $x<0$
    (counted with multiplicity). We denote them by
    $-\xi _4<-\xi _3<-\xi _2<-\xi _1$; if they are $6$, then between
    two consecutive negative roots of $V$ there are three critical points of
    which we choose the rightmost one. We denote the corresponding
    critical values by
    $\eta _i$, where $\eta _4>0$, $\eta _3<0$, $\eta _2>0$ and $\eta _1<0$.

    Suppose that $\eta _3=\eta _1$. Then we add to $V$ a positive constant (this
    does not change the sign pattern) to obtain a polynomial with multiplicity
    vector $(1,2,2)$.

    Suppose that $\eta _3<\eta _1$. We consider the family of polynomials
    $V_t:=V+tx$, $t\geq 0$ in which the sign of the coefficient of $x$ is $+$.
    As $t$ increases, the critical value $\eta _4$
    decreases faster than $\eta _3$ and $\eta _2$ decreases faster than
    $\eta _1$. Denote by $t_0>0$ the smallest value of $t$ for
    which one of the following things happens:
    \vspace{1mm}
    
    1) One has $\eta _4=\eta _1$. In this case we add to $V_{t_0}$ a positive
    constant to obtain a polynomial with multiplicity vector $(2,1,2)$. 
    \vspace{1mm}

    2) One has $\eta _2=\eta _1$. Then $-\xi _1$ is a degenerate critical point
    of $V_{t_0}$. We add to $V_{t_0}$ a positive constant and get a polynomial
    with multiplicity vector $(1,1,3)$.
    \vspace{1mm}
    
    Suppose that $\eta _3>\eta _1$. Then for $t=t_0$, one of the following
    things can take place in the family $V_t$:
    \vspace{1mm}

    3) One has $\eta_3=\eta _1$. We add to $V_{t_0}$ a positive constant and
    obtain a polynomial with multiplicity vector $(1,2,2)$.
    \vspace{1mm}

    4) One has $\eta _4=\eta _3$. Then $-\xi _3$ is a degenerate critical point
    of $V_{t_0}$. We add to $V_{t_0}$ a positive constant and get a polynomial
    with multiplicity vector $(3,1,1)$.
    \vspace{1mm}

    If 1) and 2) (resp. 3) and 4)) take place simultaneously, then the
    multiplicity vector of $V_{t_0}$ is $(2,3)$ (resp. $(3,2)$). It is not
    possible for $t=t_0$ to obtain an equality between $\eta _1$, $\eta _2$,
    $\eta _3$ or $\eta _4$ and one of the two possible other critical
    values of $V_{t_0}$ (if $V_{t_0}$ has $6$ and not $4$ negative critical
    points), because then one can add a positive constant to $V_{t_0}$
    and get a polynomial with $7$ negative roots, one conjugate pair and
    sign pattern $\Sigma _{1,7,2}$, $\Sigma _{1,6,3}$ or $\Sigma _{2,6,2}$
    which by Lemma~\ref{lmnotexist9} is impossible.

     \end{proof}

  \begin{prop}
    There exists no real monic degree $9$ polynomial having $5$ negative roots
    with multiplicity vector having $1$, $2$ or $3$ components, from the
    closure of the sign pattern 
    $\Sigma _{1,7,2}$, $\Sigma _{1,6,3}$ or $\Sigma _{2,6,2}$, and having one
    double positive root.
  \end{prop}

\begin{proof}
  A) We explain the method of proof on the example of the multiplicity vector
  $(5)$ (one five-fold negative root). We want to prove the non-existence
  of a polynomial

  $$H_5:=(x+1)^5((x+u)^2+v)(x-c)^2~,~~~\, v>0~,~~~\, c>0~,~~~\,
  u\in \mathbb{R}~.$$
  One rescales the $x$-axis to make the negative root equal to $-1$.
  The index $5$ corresponds to the multiplicity vector $(5)$.  
  Hence the non-existence of such a polynomial $H_5:=\sum _{j=0}^9a_jx^j$
  is tantamount to the
  emptyness of the domain $E_5\cap D_5$, where
  $D_5\subset \mathbb{R}^3$ is defined by the inequalities

$$D_5~:~\{ ~(u, v, c)~|~v>0~,~c>0~\} $$
  and $E_5$ is the closed domain in $\mathbb{R}^3$ defined by the
  condition the signs of the coefficients of $H_5$ to
  correspond to the closure of one of the sign patterns $\Sigma _{1,7,2}$,
  $\Sigma _{1,6,3}$ or $\Sigma _{2,6,2}$. We remind that each coefficient of
    $H_5$ is a polynomial in the variables $(u,v,c)$, with 
  $a_0=c^2(v+u^2)$, $\ldots$, $a_8=5+2u-2c$. If the interior
  of the domain $E_5$
  is non-empty, then $E_5$ has a priori the structure of a stratified
  manifold. Its stratum of maximal dimension corresponds to polynomials
  defining the sign pattern $\Sigma _{1,7,2}$, $\Sigma _{1,6,3}$
  or $\Sigma _{2,6,2}$.
    The closure $\bar{D}_5$ of the set $D_5$ is defined by the inequalities
    $v\geq 0$, $c\geq 0$ and its border
    $\partial D_5$ by $v\geq 0$, $c\geq 0$, $cv=0$.

    No polynomial $H_5\in \partial D_5$ is from the closure of the
    sign pattern $\Sigma _{1,7,2}$, $\Sigma _{1,6,3}$ or $\Sigma _{2,6,2}$.
    Indeed, for $c=0<v$ or
    $c>0=u=v$, one has
    $H_5=x^2H^1_5$, where $H^1_5(0)>0$.
    Then the (generalized) sign pattern of $H_5$
    cannot be adjacent to $\Sigma _{1,7,2}$ or $\Sigma _{2,6,2}$,
    but only to $\Sigma _{1,6,3}$
    and the one of $H^1_5$ is adjacent to $\Sigma _{1,6,1}$.
    By Lemma~\ref{lmnotexist7} such a polynomial $H^1_5$ does not exist.
    For $v=0<c$, $u>0$, non-existence of $H_5$ follows from
    Lemma~\ref{lmnotexist9}. For $v=0<c$, $u<0$, the polynomial $H_5$
    has four
    positive roots (counted with multiplicity) which by Descartes' rule of
    signs requires at least four sign changes in the (generalized) sign
    pattern -- a contradiction. Thus $c=0$ or $v=0$ is impossible.
    \vspace{1mm}
    
    B) Next we consider the two subdomains $D_5^+:=\bar{D}_5\cap \{ u\geq 0\}$
    and
    $D_5^-:=\bar{D}_5\cap \{ u\leq 0\}$. They are convex. To show that
    no polynomial
    $H_5\in D_5^+$ (resp. $H_5\in D_5^-$) has the necessary (generalized)
    sign pattern we consider the planes $T_r^+:u+v+c=r$ (resp.
    $T_r^-:-u+v+c=r$), $r\in \mathbb{R}$. It is clear that for $r<0$, one has
    $T_r^{\pm}\cap D_5^{\pm}=\emptyset$ and $T_0^{\pm}\cap D_5^{\pm}=\{ (0,0,0)\}$,
    with $(0,0,0)\in \partial D_5$.

    We suppose that there exists a polynomial $H_5\in D_5^{\pm}$
    having the necessary (generalized) sign pattern. Then it is not in
    $\partial D_5^{\pm}$ and 
    belongs to some plane $T_r^{\pm}$ for some $r=r_0>0$. At the point $(0,0,0)$
    at least one coefficient $a_j$ of $H_5$ has the wrong sign.
    The least possible
    value of $r_0$ is the least one for which $H_5\in E_5$, i.~e.
    where the signs of all coefficients
    correspond to the closure of the sign pattern $\Sigma _{1,7,2}$,
    $\Sigma _{1,6,3}$ or
    $\Sigma _{2,6,2}$. So for $r=r_0$, at least one of the coefficients $a_j$ of
    $H_5$ vanishes, because the corresponding polynomial(s) $H_5$ belong to the
    border, but not to the interior of the set~$E_5$.
    
    We use the method of Lagrange's multipliers as follows. We are looking for
    the minimal value of the function $T_0^{\pm}$ on the hypersurface
    $\{ a_j=0\}$. 
    We construct the function

    $$\tilde{T}^{\pm}_j:=\pm u+v+c+\lambda a_j~,$$
    where
    $\lambda$ is a Lagrange multiplier. For $j=0$, $\ldots$, $8$,
    we consider the system of equations

    $$\partial \tilde{T}^{\pm}_j/\partial \lambda =
    a_j=\partial \tilde{T}^{\pm}_j/\partial u=
    \partial \tilde{T}^{\pm}_j/\partial v=
    \partial \tilde{T}^{\pm}_j/\partial c=0~.$$
    In each case we are looking for a solution with $\lambda \in \mathbb{R}$
    and $(u,v,c)\in D_5^{\pm}$. It turns out that in each case either there is
    no solution or the sign of $u$, $v$ or $c$ of the solution is not the
    right one or the signs of the coefficients of the obtained polynomial
    $H_5$ are in contradiction with the closure of the sign pattern
    $\Sigma _{1,7,2}$, $\Sigma _{1,6,3}$ or $\Sigma _{2,6,2}$. This can be
    established using computer algebra.
    \vspace{1mm}

     C) The multiplicity vector $(5)$ is adjacent to four multiplicity vectors
    with two components: $(4,1)$, $(3,2)$, $(2,3)$ and $(1,4)$. We prove the
    non-existence of polynomials

    $$H_{k,5-k}:=(x+\mu )^k(x+1)^{5-k}((x+u)^2+v)(x-c)^2~,~~~\, k=1,~2,~3,~4~,$$
    with $\mu >1$ and with $(u,v,c)$ as above. Hence

    $$\mathbb{R}^4\supset D_{k,5-k}=\{ (u,v,c,\mu )~|~v>0~,~c>0~,~\mu >1~\} ~.$$
    As in A) one shows that no polynomial $H_{k,5-k}\in \partial D_{k,5-k}$
    with $c=0$ or $v=0$ has signs of the coefficients from the closure of
    $\Sigma _{1,7,2}$, $\Sigma _{1,6,3}$ or $\Sigma _{2,6,2}$.
    For $\mu =1$, one is looking in fact
    for a polynomial $H_5$ about which it was shown in A) -- B) that it
    does not exist. Hence $E_{k,5-k}\cap \partial D_{k,5-k}=\emptyset$.
    \vspace{1mm}

    D) To prove that $E_{j,5-j}\cap D_{j,5-j}=\emptyset$ we use the again the
    method of Lagrange multipliers. We set $D_{k,5-k}^+:=D_{k,5-k}\cap \{ u>0\}$, 
    $D_{k,5-k}^-:=D_{k,5-k}\cap \{ u<0\}$, $S_r^{\pm}:\pm u+v+c+(\mu -1)=r$
    ($r\in \mathbb{R}$) and
    $\tilde{S}_j^{\pm}:=\pm u+v+c+(\mu -1)+\lambda a_j$. In this case

    $$a_8=j\mu +(5-j)+2u-2c~,~\ldots ~,~a_0=\mu ^jc^2(v+u^2)~.$$
    For $j=0$, $\ldots$, $8$, and for $k$ as above,
    we consider the system of equations

    $$\partial \tilde{S}^{\pm}_j/\partial \lambda =a_j=
    \partial \tilde{S}^{\pm}_j/\partial u=
    \partial \tilde{S}^{\pm}_j/\partial v=
    \partial \tilde{S}^{\pm}_j/\partial c=
    \partial \tilde{S}^{\pm}_j/\partial \mu =0~.$$
We are looking for a solution with $\lambda \in \mathbb{R}$
and $(u,v,c,\mu )\in D_{k,5-k}^{\pm}$. 
In each case either there is
no real solution or the sign of $u$, $v$, $c$ or $\mu -1$
of the solution is not the
    right one.
    \vspace{1mm}

    E) The possible multiplicity vectors with three components are
    $(2,2,1)$, $(2,1,2)$, $(1,2,2)$, $(3,1,1)$, $(1,3,1)$ and $(1,1,3)$.
    The polynomials $H_{j,k,5-j-k}$ defined after the multiplicity vectors
    $(j,k,5-j-k)$ are

    $$H_{j,k,5-j-k}:=(x+\mu _2)^j(x+\mu _1)^k(x+1)^{5-j-k}((x+u)^2+v)(x-c)^2~,$$
    with $(u,v,c)$ as above and $1<\mu _1<\mu _2$. We set

    $$\mathbb{R}^5\supset D_{j,k,5-j-k}:=\{ (u,v,c,\mu _1,\mu _2)~|~v>0~,~c>0~,~
    1<\mu _1~,1<\mu _2~\} ~,$$
    i.~e. we consider a domain larger than the strict analog of the domains
    $D_5$ and $D_{j,5-j}$. (The strict analog would be defined by
    $1<\mu _1<\mu _2$ instead of $1<\mu _1$, $1<\mu _2$.) This is done with
    the aim to simplify the computations. We set also 
    $D_{j,k,5-j-k}^-:=D_{j,k,5-j-k}\cap \{ u<0\}$, $K_r^{\pm}:\pm u+v+c+(\mu _1-1)=r$
    ($r\in \mathbb{R}$) and
    $\tilde{K}_i^{\pm}:=\pm u+v+c+(\mu -1)+(\mu _2-1)+\lambda a_i$.
    For $i=0$, $\ldots$, $8$, and for $(j,k)$ as above,
    we consider the system of equations

    $$\partial \tilde{K}^{\pm}_i/\partial \lambda
    =a_i=\partial \tilde{K}^{\pm}_i/\partial u=
    \partial \tilde{K}^{\pm}_i/\partial v=
    \partial \tilde{K}^{\pm}_i/\partial c=
    \partial \tilde{K}^{\pm}_i/\partial \mu _1=
    \partial \tilde{K}^{\pm}_i/\partial \mu _2=0~.$$
    We are looking for a solution with $\lambda \in \mathbb{R}$
    and $(u,v,c,\mu_1,\mu _2 )\in D_{j,k,5-j-k}^{\pm}$.
    In each case either there is
no real solution or the sign of $u$, $v$, $c$, $\mu _1-1$ or $\mu _2-1$ 
of the solution is not the
right one or, when a solution exists, the obtained polynomial $H_{j,k,5-j-k}$
does not have a (generalized) sign pattern from the closure of
$\Sigma _{1,7,2}$, $\Sigma _{1,6,3}$ or $\Sigma _{2,6,2}$.

\end{proof}


\end{document}